\documentclass[11pt]{article}
\usepackage{amsmath,amssymb,amsthm,color}
\usepackage{comment}

\newtheorem{theorem}{Theorem}[section]
\newtheorem{corollary}{Corollary}[section]
\newtheorem{lemma}{Lemma}[section]
\newtheorem{proposition}{Proposition}[section]
\newtheorem{definition}{Definition}[section]
\newtheorem{example}{Example}[section]
\newtheorem{remark}{Remark}[section]

\def\R{{\mathfrak R}\, }
\def\D{{\mathfrak D}\, }

\begin{document}
\begin{center}{\bf \LARGE Lie Maps on Alternative Rings}\\
\vspace{.2in}
\noindent {\bf Bruno Leonardo Macedo Ferreira}\\
{\it Technological Federal University of Paran\'{a},\\
Professora Laura Pacheco Bastos Avenue, 800,\\
85053-510, Guarapuava, Brazil.}\\
e-mail: brunoferreira@utfpr.edu.br
\\
and
\\
\noindent {\bf Henrique Guzzo Jr.}\\
{\it Institute of Mathematics and Statistics\\
University of S\~{a}o Paulo,\\
Mat\~{a}o Street, 1010,\\   
05508-090, S\~{a}o Paulo, Brazil.}\\
e-mail: guzzo@ime.usp.br

\end{center}
\begin{abstract} In this paper we study to alternative rings the almost additivity of the Lie multiplicative and Lie triple derivable maps.
\end{abstract}
{\bf {\it AMS 2010 Subject Classification:}} 17A36, 17D05\\
{\bf {\it Keywords:}} Additivity; Lie multiplicative maps; Lie triple derivable maps; Prime alternative rings\\
{\bf {\it Running Head:}} Lie Multiplicative and Lie triple derivable maps
\section{Alternative rings and Lie multiplicative maps}

Let $\mathfrak{R}$ be a ring not necessarily associative or commutative and consider the following convention for its multiplication operation: $xy\cdot z = (xy)z$ and $x\cdot yz = x(yz)$ for $x,y,z\in \mathfrak{R}$, to reduce the number of parentheses. We denote the {\it associator} of $\mathfrak{R}$ by $(x,y,z)=xy\cdot z-x\cdot yz$ for $x,y,z\in \mathfrak{R}$. And $[x,y] = xy - yx$ is the usual Lie product of $x$ and $y$, with $x,y \in \mathfrak{R}$.

Let $\mathfrak{R}$ and $\mathfrak{R}'$ be two rings and $\varphi:\mathfrak{R}\rightarrow \mathfrak{R}'$ a map of $\mathfrak{R}$ in $\mathfrak{R}'$. We call $\varphi$ a {\it Lie multiplicative map} of $\mathfrak{R}$ in $\mathfrak{R}'$ if for all $x, y \in \mathfrak{R}$ 
\begin{eqnarray*}
\varphi\big([x,y]) = [\varphi(x),\varphi(y)].
\end{eqnarray*}
And let $\mathfrak{R}$ be a ring and $\D:\mathfrak{R}\rightarrow \mathfrak{R}$ a map of $\mathfrak{R}$ into itself. We call $\D$ a {\it Lie triple derivable multiplicative map} of $\mathfrak{R}$ into itself if for all $x, y, z \in \mathfrak{R}$ 
\begin{eqnarray*}
\D\big([[x,y],z]) = [[\D(x),y],z] + [[x,\D(y)],z] + [[x,y],\D(z)].
\end{eqnarray*}
And if $\D([x,y]) = [\D(x), y] + [x,\D(y)]$ for all $x, y \in \R$ we say that $\D : \R \rightarrow \R$ is a \textit{Lie derivable multiplicative map}. 

A ring $\mathfrak{R}$ is said to be {\it alternative} if $(x,x,y)=0=(y,x,x)$ for all $x,y\in \mathfrak{R}$. One easily sees that any associative ring is an alternative ring. An alternative ring $\mathfrak{R}$ is called {\it k-torsion free} if $k\,x=0$ implies $x=0,$ for any $x\in \mathfrak{R},$ where $k\in{\mathbb Z},\, k>0$, and {\it prime} if $\mathfrak{AB} \neq 0$ for any two nonzero ideals $\mathfrak{A},\mathfrak{B}\subseteq \mathfrak{R}$.
The {\it nucleus} of an alternative ring $\mathfrak{R}$ \mbox{is defined by} $$\mathcal{N}(\mathfrak{R})=\{r\in \mathfrak{R}\mid (x,y,u)=0=(x,u,y)=(u,x,y) \hbox{ for all }x,y\in \mathfrak{R}\}.$$
And the {\it centre} of an alternative ring $\mathfrak{R}$ \mbox{is defined by} $$\mathcal{Z}(\mathfrak{R})=\{r\in \mathcal{N}\mid [r, x] = 0 \hbox{ for all }x \in \mathfrak{R}\}.$$
\begin{theorem}\label{meu}
Let $\mathfrak{R}$ be a $3$-torsion free alternative ring. So
$\mathfrak{R}$ is a prime ring if and only if $a\mathfrak{R} \cdot b=0$ (or $a \cdot \mathfrak{R}b=0$) implies $a = 0$ or $b =0$ for $a, b \in \mathfrak{R}$. 
\end{theorem}
\begin{proof}
Clearly all alternative rings satisfying the properties $a\mathfrak{R} \cdot b=0$ (or $a \cdot \mathfrak{R}b=0$) are prime rings.
Suppose $\mathfrak{R}$ is a prime ring by \cite[Lemma $2.4$, Theorem $A$ and Proposition $3.5$]{slater} we have $\mathfrak{R} = \mathcal{A}_{0} \supseteq \mathcal{A}_{1} \supseteq \cdots \supseteq \mathcal{A}_{n} = \mathcal{A} \neq (0)$ is a chain of subrings of $\mathfrak{R}$. If $a\mathfrak{R} \cdot b=0$ (or $a \cdot \mathfrak{R}b=0$) hence $a\mathcal{A} \cdot b=0$ (or $a \cdot \mathcal{A}b=0$) follows \cite[Proposition 3.5 (e)]{slater} that $a= 0$ or $b=0$. 
 \end{proof}

\begin{definition}
A ring $\R$ is said to be flexible if satisfies
$$(x,y,x) = 0 \ \ for \ all \ x,y \in \R.$$
It is known that alternative rings are flexible.
\end{definition}

\begin{proposition}\label{Lflexivel}
Let $\R$ be a alternative ring then $\R$ satisfies 
$$(x,y,z) + (z,y,x) = 0 \ \ for \ \ all \ x,y,z \in \R.$$
\end{proposition}
\begin{proof} Just linearize the identity $(x,y,x) = 0$.
\end{proof}
A nonzero element $e_{1}\in \mathfrak{R}$ is called an {\it idempotent} if $e_{1}e_{1}=e_{1}$ and a {\it nontrivial idempotent} if it is an idempotent different from the multiplicative identity element of $\mathfrak{R}$. Let us consider $\mathfrak{R}$ an alternative ring and fix a nontrivial idempotent $e_{1}\in\mathfrak{R}$. Let \mbox{$e_2 \colon\mathfrak{R}\rightarrow\mathfrak{R}$} and $e'_2 \colon\mathfrak{R}\rightarrow\mathfrak{R}$ be linear operators given by $e_2(a)=a-e_1a$ and $e_2'(a)=a-ae_1.$ Clearly $e_2^2=e_2,$ $(e_2')^2=e_2'$ and we note that if $\mathfrak{R}$ has a unity, then we can consider $e_2=1-e_1\in \mathfrak{R}$. Let us denote $e_2(a)$ by $e_2a$ and $e_2'(a)$ by $ae_2$. It is easy to see that $e_ia\cdot e_j=e_i\cdot ae_j~(i,j=1,2)$ for all $a\in \mathfrak{R}$. Then $\mathfrak{R}$ has a Peirce decomposition
$\mathfrak{R}=\mathfrak{R}_{11}\oplus \mathfrak{R}_{12}\oplus
\mathfrak{R}_{21}\oplus \mathfrak{R}_{22},$ where
$\mathfrak{R}_{ij}=e_{i}\mathfrak{R}e_{j}$ $(i,j=1,2)$ \cite{He}, satisfying the following multiplicative relations:
\begin{enumerate}\label{asquatro}
\item [\it (i)] $\mathfrak{R}_{ij}\mathfrak{R}_{jl}\subseteq\mathfrak{R}_{il}\
(i,j,l=1,2);$
\item [\it (ii)] $\mathfrak{R}_{ij}\mathfrak{R}_{ij}\subseteq \mathfrak{R}_{ji}\
(i,j=1,2);$
\item [\it (iii)] $\mathfrak{R}_{ij}\mathfrak{R}_{kl}=0,$ if $j\neq k$ and
$(i,j)\neq (k,l),\ (i,j,k,l=1,2);$
\item [\it (iv)] $x_{ij}^{2}=0,$ for all $x_{ij}\in \mathfrak{R}_{ij}\ (i,j=1,2;~i\neq j).$
\end{enumerate}

The first result about the additivity of maps on rings was given by Martindale III \cite{Mart}. He established a condition on a ring $\mathfrak{R}$ such that every multiplicative isomorphism on $\mathfrak{R}$ is additive. Ferreira and Ferreira \cite{Fer} also
considered this question in the context of $n$-multiplicative maps on alternative rings
satisfying Martindale's conditions. They proved the following theorems.

\begin{theorem}\cite[Theorem 2.2.]{Fer} Let $\mathfrak{R}$ and $\mathfrak{R}'$ be two alternative rings. Suppose that $\mathfrak{R}$ is a ring containing a family $\{e_{\alpha}|\alpha\in \Lambda\}$ of idempotents which satisfies:
\begin{enumerate}
\item[\it (i)] If $x\in \mathfrak{R}$ is such that $x\mathfrak{R}=0,$ then $x = 0$;
\item[\it (ii)] If $x\in \mathfrak{R}$ is such that $e_{\alpha}\mathfrak{R}\cdot x = 0$ (or $e_{\alpha}\cdot \mathfrak{R}x = 0$) for all $\alpha\in \Lambda,$ then $x = 0$ (and hence $\mathfrak{R}x=0$ implies $x = 0$);
\item[\it (iii)] For each $\alpha \in \Lambda$ and $x\in \mathfrak{R},$ if $(e_{\alpha}xe_{\alpha})\cdot \mathfrak{R}(1-e_{\alpha})=0$ then $e_{\alpha}xe_{\alpha}=0$.
\end{enumerate}
Then every $n$-multiplicative isomorphism $\varphi$ of $\mathfrak{R}$ onto an arbitrary ring $\mathfrak{R}'$ is additive.
\end{theorem}

Changjing and Quanyuan \cite{chang} and Changjing et al. considered also the investigation of the almost additivity of maps for
the case of Lie multiplicative maps and Lie triple derivable maps on associative rings. They proved the
following theorem.

\begin{theorem}\label{cha}
Let $\mathfrak{R}$ be a ring containing a nontrivial idempotent $P$ and satisfying the
following condition: $(\mathbb{Q})$ If $A_{11}B_{12} = B_{12}A_{22}$ for all $B_{12} \in \mathfrak{R}_{12}$, then $A_{11} + A_{22} \in \mathcal{Z}(\mathfrak{R})$. Let $\mathfrak{R}'$ be another ring. Suppose that a bijection map $\Phi: \mathfrak{R} \rightarrow \mathfrak{R}'$ satisfies
$$\Phi([A,B]) = [\Phi(A),\Phi(B)]$$
for all $A,B \in \mathfrak{R}$. Then $\Phi(A + B) = \Phi(A) + \Phi(B) + Z_{A,B}'$ for all $A,B \in \mathfrak{R}$, where $Z_{A,B}'$ is an element in the centre $\mathcal{Z}(\mathfrak{R}')$ of $\mathfrak{R}'$ depending on $A$ and $B$.
\end{theorem}

and

\begin{theorem}\label{chad}
Let $\mathfrak{R}$ be a ring containing a nontrivial idempotent $P$ and satisfying the
following condition: $(\mathbb{Q})$ If $A_{11}B_{12} = B_{12}A_{22}$ for all $B_{12} \in \mathfrak{R}_{12}$, then $A_{11} + A_{22} \in \mathcal{Z}(\mathfrak{R})$. Suppose that a map $\delta: \mathfrak{R} \rightarrow \mathfrak{R}$ satisfies
$$\delta([[A,B],C]) = [[\delta(A),B],C] + [[A,\delta(B)],C] + [[A,B],\delta(C)]$$
for all $A,B, C \in \mathfrak{R}$. Then there exists a $Z_{A,B}$ (depending on A and B) in $\mathcal{Z}(\mathfrak{R})$ such that $\delta (A+B) = \delta(A) + \delta(B) + Z_{A,B}$.
\end{theorem}

It is noteworthy that the types of applications and the conditions usually vary
according to each problem. 

The hypotheses of the Changjing and Quanyuan's Theorem \cite{chang} and Changjing et al. \cite{changd} allowed the author to make its proof based on calculus using the Peirce decomposition notion for associative rings.
The notion of Peirce decomposition for the alternative rings is similar to the notion of Peirce decomposition for the associative rings. However, the similarity of this notion is only in its written form, but not in its theoretical structure because the Peirce decomposition for alternative rings is the generalization of the Peirce decomposition for associative rings. Taking this fact into account, in the present paper we investigated the main Changjing and Quanyuan's Theorem \cite{chang} and Changjing's et al. Theorem \cite{changd} to the class of alternative rings. For this, we adopt and follow the same structure of the demonstration presented in \cite{chang} and \cite{changd}, in order: to preserve the author ideas and to highlight the investigation of the associative results to the alternative results.
Therefore, our lemmas and the theorem that seem to be equal in written form with the lemmas and the theorem proposed in Changjing and Quanyuan \cite{chang} and Changjing's et al. \cite{changd}, are distinguished by a fundamental item: the use of the non-associative multiplications. The symbol ``$\cdot$", as defined in the introduction section of our article, is essential to elucidate how the non-associative multiplication should be done, and also the symbol ``$\cdot$" is used to simplify the notation. Therefore, the symbol ``$\cdot$" is crucial to the logic and characterization of associative results to the alternative results.

\section{Lie Multiplicative Maps}

\subsection{Auxiliary Lemmas} 

The three lemmas that follow, have identical proofs, as in \cite{chang}  (Claim $1$, Claim $2$ and Claim $3$ ). Thus,
they will be omitted.

\begin{lemma}\label{Claim1} $\varphi(0) = 0$.
\end{lemma}

\begin{lemma}\label{Claim2} For any $a \in \mathfrak{R}$ and $z \in \mathcal{Z}(\mathfrak{R})$, there exists $z' \in \mathcal{Z}(\mathfrak{R}')$ such that $\varphi(a + z) = \varphi(a) + z'$.
\end{lemma}

\begin{lemma}\label{Claim3} Let $a,b,c \in \mathfrak{R}$ and $\varphi(c) = \varphi(a) + \varphi(b)$. Then for any $t, s \in \mathfrak{R}$, we have $\varphi([c,t]) = \varphi([a,t]) + \varphi([b,t])$ and $\varphi([[c,t],s]) = \varphi([[a,t],s]) + \varphi([[b,t],s])$.
\end{lemma}

\section{Main theorem}

We shall prove as follows the main result of this paper.

\begin{theorem}\label{mainthm} Let $\mathfrak{R}$ and $\mathfrak{R}'$ be alternative rings.
Suppose that $\mathfrak{R}$ is a ring containing a nontrivial idempotent $e_1$ which satisfies:
\begin{enumerate}
\item[\it (i)] If $[a_{11}+ a_{22}, \mathfrak{R}_{12}] = 0$, then $a_{11} + a_{22} \in \mathcal{Z}(\mathfrak{R}),$
\item[\it (ii)] If $[a_{11}+ a_{22}, \mathfrak{R}_{21}] = 0$, then $a_{11} + a_{22} \in \mathcal{Z}(\mathfrak{R}).$
\end{enumerate}
Then every Lie multiplicative bijection $\varphi$ of $\mathfrak{R}$ onto an arbitrary alternative ring $\mathfrak{R}'$ is almost additive.
\end{theorem}

The following lemmas has the same hypotheses of Theorem \ref{mainthm} and we need these lemmas for the proof of this theorem. Thus, let us consider $e_{1}$ a nontrivial idempotent of $\mathfrak{R}$. 

\begin{lemma}\label{lema4} For any $a_{11} \in \R_{11}$, $b_{ij} \in \R_{ij}$, with $i \neq j$ there exist $Z'_{a_{11}, b_{ij}} \in \mathcal{Z}(\R')$ such
that
$$\varphi(a_{11} + b_{ij}) = \varphi(a_{11}) + \varphi(b_{ij}) + Z'_{a_{11}, b_{ij}}.$$
\end{lemma}
\begin{proof}
We shall only prove the case $i=2$, $j=1$ because the demonstration of the other case is similar. 
By surjectivity of $\varphi$ there exist $c = c_{11} + c_{12} + c_{21} + c_{22} \in \R$ such that $\varphi(c) = \varphi(a_{11}) + \varphi(b_{21})$. Applying the Lemma \ref{Claim1} and \ref{Claim3} we have
$$\varphi([c,e_1]) = \varphi([a_{11}, e_1]) + \varphi([b_{21}, e_1]) = \varphi([b_{21}, e_1]).$$
Since $\varphi$ is injective, we get $[c, e_1] = [b_{21}, e_1]$. Thus $c_{21} = b_{21}$ and $c_{12} = 0$.
Now for any $x_{12} \in \R_{12}$, we have
$$\varphi([[c,x_{12}],e_1]) = \varphi([[a_{11},x_{12}],e_1]) + \varphi([[b_{21},x_{12}],e_1]).$$
By the injectivity of $\varphi$ and Lemma \ref{Claim1}, we get $[c_{11} - a_{11} + c_{22}, x_{12}] = 0$. Therefore by condition $(i)$ of the Theorem \ref{mainthm} we have
$c_{11} - a_{11} + c_{22} \in \mathcal{Z}(\R)$. And finally by
Lemma \ref{Claim2} we verified that the Lemma is valid.
\end{proof}

\begin{lemma}\label{lema5}
For any $a_{12} \in \R_{12}$ and $b_{21} \in \R_{21}$, we have $\varphi(a_{12} + b_{21}) = \varphi(a_{12}) + \varphi(b_{21})$.
\end{lemma}

\begin{proof}
By the same arguments of Claim 7 of \cite{chang}. 
\end{proof}

\begin{lemma}\label{lema6}
For any $a_{ij}, b_{ij} \in \R_{ij}$, we have $\varphi(a_{ij} + b_{ij}) = \varphi(a_{ij}) + \varphi(b_{ij})$.
\end{lemma}

\begin{proof}
Here too we shall only prove the case $i=2$, $j=1$ because the demonstration of the other case is similar.
Firstly observe that by $x_{ij}^{2}=0,$ for all $x_{ij}\in \mathfrak{R}_{ij}\ (i,j=1,2;~i\neq j)$ we have
$$a_{21} + b_{21} + 2b_{21}a_{21} = [e_1 + a_{21}, e_1 - b_{21}].$$
Now making use of Lemma \ref{lema4} and \ref{lema5} we get
\begin{eqnarray*}
\varphi(a_{21} + b_{21}) &+& \varphi(2b_{21}a_{21}) = \varphi(a_{21} + b_{21} +2b_{21}a_{21}) \\&=& \varphi([e_1 + a_{21}, e_1 - b_{21}]) \\&=& [\varphi(e_1 + a_{21}), \varphi(e_1 - b_{21})] \\&=& [\varphi(e_1) + \varphi(a_{21}) + Z'_{e_1, a_{21}}, \varphi(e_1) + \varphi(-b_{21}) + Z'_{e_1, b_{21}}] \\&=& [\varphi(e_1),\varphi(e_1)]+[\varphi(a_{21}),\varphi(e_1)]+[\varphi(e_1), \varphi(-b_{21}) ]\\&+&[\varphi(a_{21}), \varphi(-b_{21})] \\&=& \varphi([e_1, e_1]) + \varphi([a_{21}, e_1]) + \varphi([e_1, -b_{21}]) + \varphi([a_{21}, -b_{21}]) \\&=& \varphi(a_{21}) + \varphi(b_{21}) + \varphi(2b_{21}a_{21}). 
\end{eqnarray*}
For the case $i=1$, $j=2$ make use of 
$$a_{12} + b_{12} + 2a_{12}b_{12} = [e_1 - b_{12}, e_1 + a_{12}].$$
\end{proof}

\begin{lemma}\label{lema7}
For any $a_{ii}, b_{ii} \in \R_{ii}$, $i=1,2$, there exists a $Z'_{a_{ii},b_{ii}} \in \mathcal{Z}(\R')$ such that
$$\varphi(a_{ii} + b_{ii}) = \varphi(a_{ii}) + \varphi(b_{ii}) + Z'_{a_{ii},b_{ii}}.$$
\end{lemma}
\begin{proof}
By the same arguments of Claim 6 of \cite{chang}.
\end{proof}

\begin{lemma}\label{lema8}
For any $a_{11} \in \R_{11}$, $b_{12} \in \R_{12}$, $c_{21} \in \R_{21}$, $d_{22} \in \R_{22}$, there exists a $Z'_{a_{11},b_{12}, c_{21}, d_{22}} \in \mathcal{Z}(\R')$ such that
$$\varphi(a_{11} + b_{12} + c_{21} + d_{22}) = \varphi(a_{11}) + \varphi(b_{12}) + \varphi(c_{21}) + \varphi(d_{22}) + Z'_{a_{11},b_{12}, c_{21}, d_{22}}.$$
\end{lemma}
\begin{proof}
By surjectivity of $\varphi$ there exist $h = h_{11} + h_{12} + h_{21} + h_{22} \in \R$ such that $\varphi(h) = \varphi(a_{11}) + \varphi(b_{12}) + \varphi(c_{21}) + \varphi(d_{22})$. Applying the Lemma \ref{Claim1}, \ref{Claim3} and \ref{lema5} we have
\begin{eqnarray*}
\varphi([e_1, h]) &=& \varphi([e_1, a_{11}])+\varphi([e_1, b_{12}])+\varphi([e_1, c_{21}])+\varphi([e_1, d_{22}])
\\&=& \varphi(b_{12}) + \varphi(-c_{21})
\\&=& \varphi(b_{12} - c_{21}).
\end{eqnarray*}
Since $\varphi$ is injective, we get $h_{12} = b_{12}$ and $h_{21} = c_{21}$.
Now for any $x_{12} \in \R_{12}$, by Lemma \ref{lema5} and \ref{lema6} we obtain
\begin{eqnarray*}
\varphi([[h,x_{12}],e_1]) &=& \varphi([[a_{11},x_{12}],e_1]) + \varphi([[b_{12},x_{12}],e_1])
\\&+& \varphi([[c_{21},x_{12}],e_1]) + \varphi([[d_{22},x_{12}],e_1])
\\&=& \varphi(-a_{11}x_{12}) + \varphi(2b_{12}x_{12}) + \varphi(x_{12}d_{22})
\\&=& \varphi(-a_{11}x_{12} + 2b_{12}x_{12} + x_{12}d_{22}).
\end{eqnarray*}
As $\varphi$ is injective, we get $[h_{11} + h_{22} - a_{11} -d_{22}, x_{12}] = 0$ for all $x_{12} \in \R_{12}$. By condition $(i)$ of the Theorem \ref{mainthm}
we have $h = a_{11} + b_{12} + c_{21} + d_{22} + Z$ for some $Z \in \mathcal{Z}(\R)$. Thus the Lemma is true by Lemma \ref{Claim2}.
\end{proof}

\vspace{0,5cm}

We are ready to prove our Theorem \ref{mainthm}.

\vspace{0,5cm}

\noindent Proof of Theorem. Let $a, b \in \R$ with $a= a_{11} + a_{12} + a_{21} + a_{22}$ and $b= b_{11} + b_{12} + b_{21} + b_{22}$. By previous Lemmas we obtain
\begin{eqnarray*}
\varphi(a + b) &=& \varphi(a_{11} + a_{12} + a_{21} + a_{22} + b_{11} + b_{12} + b_{21} + b_{22})
\\&=& \varphi((a_{11} + b_{11}) + (a_{12} + b_{12}) + (a_{21}+ b_{21}) + (a_{22} + b_{22}))
\\&=&\varphi(a_{11} + b_{11}) + \varphi(a_{12} + b_{12}) + \varphi(a_{21}+ b_{21}) + \varphi(a_{22} + b_{22}) + Z'_{1}
\\&=& \varphi(a_{11}) +\varphi( b_{11}) + Z'_{2} + \varphi(a_{12}) + \varphi(b_{12}) + \varphi(a_{21})
\\&+& \varphi(b_{21}) +\varphi(a_{22}) + \varphi(b_{22}) + Z'_{3} + Z'_{1}
\\&=& (\varphi(a_{11}) +\varphi(a_{12}) +\varphi(a_{21})+\varphi(a_{22})) 
\\&+& (\varphi( b_{11}) + \varphi( b_{12}) + \varphi( b_{21}) + \varphi( b_{22})) + (Z'_1 + Z'_2 + Z'_3)
\\&=& \varphi(a_{11} + a_{12} + a_{21} + a_{22}) - Z'_{4} + \varphi(b_{11} + b_{12} + b_{21} + b_{22}) - Z'_{5} 
\\&+&(Z'_1 + Z'_2 + Z'_3)
\\&=& \varphi(a) + \varphi(b) + (Z'_1 + Z'_2 + Z'_3 - Z'_4 -Z'_5)
\\&=& \varphi(a) + \varphi(b) + Z'_{a,b}.
\end{eqnarray*}
It is therefore our theorem is proved.

\section{Lie triple derivable Maps}

\subsection{Main theorem}

We shall prove as follows the main result of this paper.

\begin{theorem}\label{mainthmd} Let $\mathfrak{R}$ be an alternative rings.
Suppose that $\mathfrak{R}$ is a ring containing a nontrivial idempotent $e_1$ which satisfies the same hypotheses of the Theorem \ref{mainthm}.
Then every Lie triple derivable map $\D$ of $\mathfrak{R}$ into itself is almost additive.
\end{theorem}

The following lemmas has the same hypotheses of Theorem \ref{mainthmd} and we need these lemmas for the proof of this theorem. Thus, let us consider $e_{1}$ a nontrivial idempotent of $\mathfrak{R}$. It's worth highlighting that some lemmas have their proof equal to the claims in \cite{changd} and when this occurs we will make the proper mention.
We started with the following

\begin{lemma}\label{Claim1d} $\D(0) = 0$.
\end{lemma}
\begin{proof}
This Lemma have identical proof as Claim $1$ in \cite{changd}. 
\end{proof} 

\begin{lemma}\label{lema4d} For any $a_{11} \in \R_{11}$, $b_{ij} \in \R_{ij}$, with $i \neq j$ there exist $z_{a_{11}, b_{ij}} \in \mathcal{Z}(\R)$ such
that
$$\D(a_{11} + b_{ij}) = \D(a_{11}) + \D(b_{ij}) + z_{a_{11}, b_{ij}}.$$
\end{lemma}
\begin{proof}
We shall only prove the case $i=1$, $j=2$ because the demonstration of the other case is similar just use the condition (i) of the Theorem \ref{mainthmd}.
According to Changjing et al. we considered $t = \D(a_{11} + b_{12}) - \D(a_{11}) - \D(b_{12})$.
As in the associative case we get $[[t, e_1],e_1]=0$ just to observe that 
$$\D([[a_{11} + b_{12},e_1],e_1]) = \D(b_{12}) = \D([[a_{11},e_1],e_1]) + \D([[b_{12},e_1],e_1]).$$ 
It follows that $t_{12} + t_{21} = 0$ just use the definition of $\D$. Now we will use the condition (ii) of the Theorem \ref{mainthmd}, for this let any $c_{21} \in \R_{21}$ and note that 
$$\D([[a_{11} + b_{12},c_{21}],e_1]) = \D(-c_{21}a_{11}) = \D([[a_{11},c_{21}],e_1]) + \D([[b_{12}, c_{21}],e_1]).$$
So using the definition of $\D$ and Lemma \ref{Claim1d} we obtain $[t_{11} + t_{22}, c_{21}] = 0$. Therefore by condition (ii) of the Theorem \ref{mainthmd} we have $t_{11} + t_{22} \in \mathcal{Z}(\R)$. Thus, $\D(a_{11} + b_{12}) = \D(a_{11}) + \D(b_{12}) + z_{a_{11},b_{12}}$.

\end{proof}

\begin{lemma}\label{lema5d}
For any $a_{12} \in \R_{12}$ and $b_{21} \in \R_{21}$, we have $\D(a_{12} + b_{21}) = \D(a_{12}) + \D(b_{21})$.
\end{lemma}

\begin{proof}
By the same arguments of Claim 5 of \cite{changd}. 
\end{proof}

\begin{lemma}\label{lema6d}
For any $a_{ij}, b_{ij} \in \R_{ij}$ with $i \neq j$, we have $\D(a_{ij} + b_{ij}) = \D(a_{ij}) + \D(b_{ij})$.
\end{lemma}

\begin{proof}
Here we shall only prove the case $i=2$, $j=1$ because the demonstration of the other case is similar.
Firstly observe that by $x_{ij}^{2}=0,$ for all $x_{ij}\in \mathfrak{R}_{ij}\ (i,j=1,2;~i\neq j)$ we have
$$a_{21} + b_{21} + 2a_{21}b_{21} = [[e_1 + a_{21}, e_1 - b_{21}],e_1].$$
Now making use of Lemma \ref{lema4d} and \ref{lema5d} we get
\begin{eqnarray*}
\D(a_{21} + b_{21}) &+& \D(2a_{21}b_{21}) = \D(a_{21} + b_{21} +2a_{21}b_{21}) \\&=& \D([[e_1 + a_{21}, e_1 - b_{21}],e_1]) \\&=& [[\D(e_1 + a_{21}), e_1 - b_{21}],e_1] + [[e_1+a_{21},\D(e_1 - b_{21})], e_1] \\&+& [[e_1 + a_{21}, e_1 - b_{21}], \D(e_1)] \\&=& [[\D(e_1) + \D(a_{21}) + z_{e_1, a_{21}}, e_1 -b_{21}],e_1] \\&+& [[e_1 + a_{21}, \D(e_1) + \D(-b_{21})+ z_{e_1,b_{21}}],e_1] \\&+& [[e_1 + a_{21}, e_1 - b_{21}], \D(e_1)] \\&=& \D([[e_1,e_1)],e_1)])+ \D([[e_1, -b_{21}],e_1]) + \D([[a_{21}, e_1],e_1]) \\&+& \D([[a_{21},-b_{21}],e_1]) \\&=& \D(a_{21}) + \D(b_{21}) + \D(2a_{21}b_{21}).
\end{eqnarray*}
For the case $i=1$, $j=2$ make use of 
$$a_{12} + b_{12} - 2a_{12}b_{12} = [e_1,[e_1 - b_{12}, e_1 + a_{12}]].$$
\end{proof}

\begin{lemma}\label{lema7d}
For any $a_{ii}, b_{ii} \in \R_{ii}$, $i=1,2$, there exists a $z_{a_{ii},b_{ii}} \in \mathcal{Z}(\R)$ such that
$$\D(a_{ii} + b_{ii}) = \D(a_{ii}) + \D(b_{ii}) + z_{a_{ii},b_{ii}}.$$
\end{lemma}
\begin{proof}
By the same arguments of Claim 4 of \cite{changd}.
\end{proof}

\begin{lemma}\label{lema8d}
For any $a_{11} \in \R_{11}$, $b_{12} \in \R_{12}$, $c_{21} \in \R_{21}$, $d_{22} \in \R_{22}$, there exists a $z_{a_{11},b_{12}, c_{21}, d_{22}} \in \mathcal{Z}(\R)$ such that
$$\D(a_{11} + b_{12} + c_{21} + d_{22}) = \D(a_{11}) + \D(b_{12}) + \D(c_{21}) + \D(d_{22}) + z_{a_{11},b_{12}, c_{21}, d_{22}}.$$
\end{lemma}
\begin{proof}
Before the proof of this Lemma, observe that in an alternative ring if any $x_{ij}, y_{ij} \in \R_{ij}$ with $i \neq j$ then $x_{ij}y_{ij} \in \R_{ji}$ and not necessarily $x_{ij}y_{ij} = 0$. In light of this we have a slight change in the proof of Claim $6$ made in \cite{changd}, but such a change is crucial for the result of the lemma to be valid.
According to Claim 6 in \cite{changd}, let $t = \D(a_{11} + b_{12} + c_{21} + d_{22}) - \D(a_{11}) - \D(b_{12}) - \D(c_{21}) - \D(d_{22})$. Using the definition of $\D$ and Lemma \ref{lema6d} we get $[[t,e_1],e_1] =0$, which implies $t_{12} + t_{21} = 0$.
Now for all $x_{12} \in \R_{12}$, by Lemmas \ref{lema5d}, \ref{lema6d} we have 
\begin{eqnarray*}
&&[[\D(a_{11} + b_{12} + c_{21} + d_{22}), x_{12}],e_1] + [[a_{11} + b_{12} + c_{21} + d_{22},\D(x_{12})],e_1] \\&+& [[a_{11} + b_{12} + c_{21} + d_{22},x_{12}],\D(e_1)] \\&=& \D([[a_{11} + b_{12} + c_{21} + d_{22}, x_{12}],e_1]) \\&=& \D(x_{12}d_{22} - a_{11}x_{12} - b_{12}x_{12}) \\&=& \D(x_{12}d_{22} - a_{11}x_{12}) + \D(-b_{12}x_{12}) \\&=& \D(x_{12}d_{22}) + \D(- a_{11}x_{12}) + \D(-b_{12}x_{12}) \\&=& \D([[a_{11}, x_{12}],e_1]) + \D([[b_{12}, x_{12}],e_1]) + \D([[c_{21}, x_{12}],e_1]) + \D([[d_{22}, x_{12}],e_1])\\&=& [[\D(a_{11}) + \D(b_{12}) + \D(c_{21}) + \D(d_{22}), x_{12}],e_1] \\&+& [[a_{11} + b_{12} + c_{21} + d_{22}, \D(x_{12})],e_1] \\&+& [[a_{11} + b_{12} + c_{21} + d_{22}, x_{12}],e_1].
\end{eqnarray*}
So, $[[t,x_{12}],e_1] = 0$ which implies, by condition (i) of the Theorem \ref{mainthmd}, $t_{11} + t_{22} \in \mathcal{Z}(\R)$.
Thus, $\D(a_{11} + b_{12} + c_{21} + d_{22}) = \D(a_{11}) + \D(b_{12}) + \D(c_{21}) + \D(d_{22}) + z_{a_{11},b_{12}, c_{21}, d_{22}}$ where $z_{a_{11},b_{12}, c_{21}, d_{22}} \in \mathcal{Z}(\R)$.  
\end{proof}

\vspace{0,5cm}

We are ready to prove our Theorem \ref{mainthmd}.

\vspace{0,5cm}

\noindent Proof of Theorem. Let $a, b \in \R$ with $a= a_{11} + a_{12} + a_{21} + a_{22}$ and $b= b_{11} + b_{12} + b_{21} + b_{22}$. By previous Lemmas we obtain
\begin{eqnarray*}
\D(a + b) &=& \D(a_{11} + a_{12} + a_{21} + a_{22} + b_{11} + b_{12} + b_{21} + b_{22})
\\&=& \D((a_{11} + b_{11}) + (a_{12} + b_{12}) + (a_{21}+ b_{21}) + (a_{22} + b_{22}))
\\&=&\D(a_{11} + b_{11}) + \D(a_{12} + b_{12}) + \D(a_{21}+ b_{21}) + \D(a_{22} + b_{22}) + z_{1}
\\&=& \D(a_{11}) +\D( b_{11}) + z_{2} + \D(a_{12}) + \D(b_{12}) + \D(a_{21})
\\&+& \D(b_{21}) +\D(a_{22}) + \D(b_{22}) + z_{3} + z_{1}
\\&=& (\D(a_{11}) +\D(a_{12}) +\D(a_{21})+\D(a_{22})) 
\\&+& (\D( b_{11}) + \D( b_{12}) + \D( b_{21}) + \D( b_{22})) + (z_1 + z_2 + z_3)
\\&=& \D(a_{11} + a_{12} + a_{21} + a_{22}) - z_{4} + \D(b_{11} + b_{12} + b_{21} + b_{22}) - z_{5} 
\\&+&(z_1 + z_2 + z_3)
\\&=& \D(a) + \D(b) + (z_1 + z_2 + z_3 - z_4 -z_5)
\\&=& \D(a) + \D(b) + z_{a,b}.
\end{eqnarray*}
It is therefore our theorem is proved.

\begin{corollary}
Let $\mathfrak{R}$ be an alternative rings.
Suppose that $\mathfrak{R}$ is a ring containing a nontrivial idempotent $e_1$ which satisfies:
\begin{enumerate}
\item[\it (i)] If $[a_{11}+ a_{22}, \mathfrak{R}_{12}] = 0$, then $a_{11} + a_{22} \in \mathcal{Z}(\mathfrak{R})$,
\item[\it (ii)] If $[a_{11}+ a_{22}, \mathfrak{R}_{21}] = 0$, then $a_{11} + a_{22} \in \mathcal{Z}(\mathfrak{R})$.
\end{enumerate}
Then every Lie derivable map $\D$ of $\mathfrak{R}$ into itself is almost additive.
\end{corollary}
\begin{proof}
Just note that Lie derivable maps are Lie triple derivable maps.
\end{proof}

\begin{remark}
It is worth noting that the hypothesis,
\begin{center}
 If $[a_{11}+ a_{22}, \mathfrak{R}_{21}] = 0$, then $a_{11} + a_{22} \in \mathcal{Z}(\mathfrak{R})$,
\end{center}
does not appear in the associative case because of the relations $\R_{12}\R_{12} = 0$ and $\R_{21}\R_{21} = 0$,
which in general is not true in alternative rings. 
\end{remark}
The following example shows us an associative ring in which conditions (i) and (ii) of the Theorems \ref{mainthm} and \ref{mainthmd} are not equivalent.
\begin{example} Let $\R$ be an associative ring with a idempotent $e \neq 0, 1$. Consider the multiplication table given by: 

\begin{center}
 \begin{tabular}{ | l | l | l | l | l | l | l |}
    \hline
    ⋅​$\cdot$ & $e$​ & $a_{11}$​ & $b_{11}$​ & $b_{12}$​ & $c_{21}$​ & $d_{22}$​\\ \hline
       $ e​$ & $e$​ & $a_{11}$​ & $b_{11}$​ & $b_{12}$​ & $0$​ & $0$​ \\ \hline
   $a_{11}$​ & $a_{11}$​ & $0$​ & $0$​ & $0$​ & $0$​ & $0​$\\ \hline
$b_{11}$​ & $b_{11}$​ & $0​$ & $b_{11}$​ & $b_{12}$​ & $0$​ & $0$​\\ \hline
	 $b_{12}$​ & $0$​ & $0$​ & $0$​ & $0$​ & $0$​ & $0$​ \\ \hline
	 $c_{21}$​ & $c_{21}$​ & $0$​ & $0$​ & $0$​ & $0$​ & $0$​ \\ \hline
	 $d_{22}$​ & $0$​ & $0$​ & $0$​ & $0$​ & $0$​ & $0$​ \\
    \hline
  \end{tabular}
\end{center}
Note that this ring is associative. And by a straightforward calculation it can be verified that $\R$ satisfies the condition (i) but does not satisfy the condition (ii) of the Theorems \ref{mainthm} and \ref{mainthmd}. Therefore the conditions of the Theorems \ref{mainthm} and \ref{mainthmd} are not equivalent.
\end{example}

Now, the following example is an alternative ring that is not associative and satisfies the hypotheses of Theorem \ref{mainthm} and \ref{mainthmd}, which allows us to show that the conditions stated in the Theorems \ref{mainthm} and \ref{mainthmd} do not represent artificial conditions.

\begin{example}
Let $\R$ be an alternative ring with a idempotent $e \neq 0, 1$. Consider the multiplication table given by: 

\begin{center}
  \begin{tabular}{ | l | l | l | l | l | l |}
    \hline
    $\cdot$ & $e$ & $a_{11}$ & $b_{12}$ & $c_{21}$ & $d_{22}$\\ \hline
        $e$ & $e$ & $a_{11}$ & $b_{12}$ & $0$ & $0$ \\ \hline
   $a_{11}$ & $a_{11}$ & $a_{11}$ & $0$ & $0$ & $0$ \\ \hline
	 $b_{12}$ & $0$ & $0$ & $0$ & $a_{11}$ & $0$ \\ \hline
	 $c_{21}$ & $c_{21}$ & $0$ & $d_{22}$ & $0$ & $0$ \\ \hline
	 $d_{22}$ & $0$ & $0$ & $0$ & $0$ & $0$ \\
    \hline
  \end{tabular}
\end{center}
Note that this ring is not associative because $(b_{12}, c_{21}, a_{11}) \neq 0$. And by a direct calculation it can be verified that $\R$ satisfies the conditions of the Theorems \ref{mainthm} and \ref{mainthmd}.
Therefore every Lie multiplicative map of $\mathfrak{R}$ in $\mathfrak{R}'$ and Lie triple derivable multiplicative map of $\mathfrak{R}$ into itself is almost additive.
\end{example}

\section{Prime alternative rings}
In this section, we shall show that prime alternative rings satisfies the conditions of the Theorems \ref{mainthm} and \ref{mainthmd}.

\begin{lemma}\label{ultlemafd}
Let $\R$ be a $3$-torsion free prime alternative ring with a nontrivial idempotent $e_1$ and $\mathcal{Z}(\R)$ be its centre.
\begin{enumerate}
\item[\it (i)] If $[a_{11} + a_{22}, \R_{12}] = 0$, then $a_{11} + a_{22} \in \mathcal{Z}(\R)$,
\item[\it (ii)] If $[a_{11} + a_{22}, \R_{21}] = 0$, then $a_{11} + a_{22} \in \mathcal{Z}(\R)$.
\end{enumerate}
\end{lemma}
\begin{proof}
We will only prove (i) because (ii) it is similar. First note that the identities are valid in alternative rings by Proposition \ref{Lflexivel}
\begin{enumerate}
\item[\it(i)] $(x_{11}, x_{12}, a_{22}) = 0 = (a_{11}, x_{11}, x_{12})$;
\item[\it(ii)] $(x_{12},x_{22},a_{22}) = 0 = (a_{11}, x_{12}, x_{22})$;
\item[\it(iii)] $(a_{22}, x_{21}, x_{12}) = 0 = (x_{21},a_{11},x_{12})$.
\end{enumerate}
Taking these identities into account we have
\begin{enumerate}
\item[\it(a)] $(a_{11}x_{11})x_{12} = a_{11}(x_{11}x_{12}) = (x_{11}x_{12})a_{22} = x_{11}(x_{12}a_{22}) = x_{11}(a_{11}x_{12}) = (x_{11}a_{11})x_{12};$

\item[\it(b)] $x_{12}(x_{22}a_{22}) = (x_{12}x_{22})a_{22} = a_{11}(x_{12}x_{22}) = (a_{11}x_{12})x_{22} = (x_{12}a_{22})x_{22} = x_{12}(a_{22}x_{22});$

\item[\it (c)] $(a_{22}x_{21})x_{12} = (a_{22}x_{21})x_{12} = (x_{21}x_{12})a_{22} = x_{21}(x_{12}a_{22}) = x_{21}(a_{11}x_{12}) = (x_{21}a_{11})x_{12},$

\end{enumerate}
for all $x_{12} \in \R_{12}$. As $\R$ is a $3$-torsion free prime alternative ring, by Theorem \ref{meu} we get
\begin{enumerate}
	\item $a_{11}x_{11} = x_{11}a_{11}$;
	\item $a_{22}x_{22} = x_{22}a_{22}$;
	\item $a_{22}x_{21} = x_{21}a_{11}.$
\end{enumerate}

Therefore for any $x \in \R$ with $x = x_{11} + x_{12} + x_{21} + x_{22}$, we obtain $[a_{11} + a_{22} , \R] = 0$. 
\end{proof}

As a last result of our paper follows the Corollaries, by Theorems \ref{mainthm} and \ref{mainthmd} and Lemma \ref{ultlemafd}. 

\begin{corollary}
Let $\mathfrak{R}$ be a $3$-torsion free prime alternative ring and $\mathfrak{R}'$ be another alternative
ring.
Suppose that $\mathfrak{R}$ is an alternative ring containing a nontrivial idempotent $e_1$. Then every Lie multiplicative bijection $\varphi$ of $\mathfrak{R}$ onto an arbitrary alternative ring $\mathfrak{R}'$ is almost additive.
\end{corollary}

\begin{corollary}
Let $\mathfrak{R}$ be a $3$-torsion free prime alternative ring. Suppose that $\mathfrak{R}$ is an alternative ring containing a nontrivial idempotent $e_1$. Then every Lie triple derivable multiplicative map $\D$ of $\mathfrak{R}$ into itself is almost additive.
\end{corollary}

\end{document}